\documentclass[reqno,10.5pt]{amsart}

\usepackage{xcolor}
\usepackage{graphicx}
\usepackage[hidelinks]{hyperref}
\usepackage{amscd}

\usepackage[T1]{fontenc}
\usepackage[utf8]{inputenc}

\usepackage{amsmath}
\usepackage{amssymb, amsmath}
\usepackage{amsthm,enumerate}
\usepackage{amssymb}
\usepackage{latexsym,array}
\usepackage{amsfonts}
\usepackage{shadow}
\usepackage{amsbsy}
\usepackage{amssymb}
\usepackage{dsfont}
\usepackage{mathtools}
\usepackage{enumitem}
\usepackage{doi}
\usepackage[utf8]{inputenc}

\usepackage[notref, notcite, final]{showkeys}
\usepackage{color}
\usepackage{graphicx}
\usepackage[hidelinks]{hyperref}
\usepackage{cleveref}
\usepackage{fullpage}
\usepackage{comment}
\usepackage{tikz-cd}
\usepackage{tikz}
\usepackage{float}

\usepackage{dsfont}

\usepackage{cleveref}

\numberwithin{equation}{section}

\newcommand{\bfb}{\mathbf{b}}
\newcommand{\bfk}{\mathbf{k}}
\newcommand{\bfy}{\mathbf{y}}
\newcommand{\bfw}{\mathbf{w}}
\newcommand{\bfe}{\mathbf{e}}
\newcommand{\bfz}{\mathbf{z}}

\newtheorem{theorem}{Theorem}[section]
\newtheorem{lemma}[theorem]{Lemma}
\newtheorem{corollary}[theorem]{Corollary}
\newtheorem{proposition}[theorem]{Proposition}

\theoremstyle{definition}
\newtheorem{remark}[theorem]{{\bf Remark}}
\newtheorem{definition}[theorem]{Definition}

\newcommand{\pp}{\partial}
\newcommand{\rr}{\mathbb{R}}

\newcommand{\vm}{{\mathbf{m}}}
\newcommand{\vx}{{\mathbf{x}}}

\newcommand{\vp}{{\mathbf{p}}}
\newcommand{\vs}{{\mathbf{s}}}

\newcommand{\bfx}{\mathbf{x}}
\newcommand{\bfm}{\mathbf{m}}

\crefname{enumi}{}{}
\crefname{enumii}{}{}

\title[]{Entire monogenic functions of given proximate order\\ and continuous homomorphisms}

\author[F. Colombo]{Fabrizio Colombo}
\address{(F.C.)
Politecnico di Milano\\Dipartimento di Matematica\\Via E. Bonardi, 9\\20133
Milano, Italy}
\email{fabrizio.colombo@polimi.it}

\author[R.S. Krausshar]{R.S. Krausshar}
\address{(R.S.K) Chair of Mathematics\\
University of Erfurt\\ Nordh\"auser Stra{\ss}e 63\\ 99089 Erfurt \\Germany }
\email{soeren.krausshar@uni-erfurt.de }

\author[S. Pinton]{Stefano Pinton}
\address{(S.P.)
Politecnico di Milano\\Dipartimento di Matematica\\Via E. Bonardi, 9\\20133
Milano, Italy}
\email{stefano.pinton@polimi.it}

\author[I. Sabadini]{Irene Sabadini}
\address{(I.S.)
Politecnico di Milano\\Dipartimento di Matematica\\Via E. Bonardi, 9\\20133
Milano, Italy}
\email{irene.sabadini@polimi.it}

\begin{document}
	
	\maketitle
	
\begin{abstract}
Infinite order differential operators appear in different fields of mathematics and physics. In the past decade they turned out to play a crucial role in the theory of superoscillations and provided new insight in the study of the evolution as initial data for the Schr\"odinger equation.
Inspired by the infinite order differential operators arising in quantum mechanics, in this paper we investigate
the continuity of a class of infinite order differential operators acting on spaces of entire hyperholomorphic functions. Precisely, we consider homomorphisms acting on functions in the kernel of the Dirac operator.
For this class of functions, often called monogenic functions, we introduce the proximate order and prove some fundamental properties.
As important application we are able to characterize infinite order differential
operators that act continuously on spaces of monogenic entire functions.
\end{abstract}
\vskip 1cm

\date{today}
\tableofcontents

\section{Introduction}\label{INTROD1}

Infinite order differential operators have been studied already since a long time.
In the recent years they turned out to be of fundamental importance in the study of the evolution of superoscillations as initial data for the Schr\"odinger  equation.
Superoscillatory functions arise in several areas of science and technology,
for example in quantum mechanics they are the outcome
of Aharonov's weak values, see \cite{aav,abook}.
To study their time evolution as initial data of quantum field equations represents an important problem in quantum mechanics.

The study of the evolution of superoscillatory functions under the Schr\"odinger equation
is highly non-trivial. A natural functional analytic setting is
the space of entire functions with certain growth conditions.
In fact, the Cauchy problem for the Schr\"odinger equation with superoscillatory initial data leads
to  infinite order differential operators of the  type
$$
\mathcal{U}(t,z;\partial_z)=\sum_{m=1}^\infty u_m(t,z)\partial_z^m,
$$
where the coefficients $u_m(t,z)$ depend on the Green's function
of the time dependent Schr\"odinger equation for a given potential $V$,
 $t$ is the time variable and $z$ is the complexification of the space variable.
For more details, see for example the monograph \cite{acsst5} and \cite{ABCS19,acsst3,acsst6,KGField,Jussi,genHYPAA,genHYP,PETER}.

\medskip
Another application of infinite order differential operators
within the scope of the theory of superoscillatory functions
is the extension of this theory to several superoscillating variables, see \cite{JFAA}.

\medskip
For $p\geq 1$ the natural spaces on which operators such as
 $\mathcal{U}(t,z;\partial_z)$ act
are the spaces of entire functions with growth order either order lower than $p$ or equal to $p$ and finite type. In other words, they  consist of entire functions $f$ for which there exist constants $B, C >0$ such that
$
|f(z)|\leq C e^{B|z|^p}.
$

\medskip
The problem to extend infinite order differential operators
to the hypercomplex setting is treated
in the recent paper \cite{ANMATHPHY} where the authors investigate the continuity of a class of infinite order differential operators
acting on spaces of entire hyperholomorphic functions
that include monogenic functions.
In  \cite{ANMATHPHY} we find the hypercomplex version of some results obtained in \cite{ACSS18,QS2}; the class of monogenic functions is the most delicate case to investigate.
Even though the classical exponential function is not in the kernel of the Dirac operator, monogenic functions with exponential bounds play a crucial role in the study of continuity of a class of infinite order differential operators in the hypercomplex settings.
The hypercomplex setting  is non-trivial and requires
some efforts because of the structure of monogenic functions
and of the fact that they admit series expansions in terms of
  the so called Fueter polynomials.
 Precisely, the Fueter's polynomials $V_k(x)$, see for instance formula $(4)$ of \cite{AlmeidaKra2005} page $794$ or elsewhere, are
  defined by
$$
V_k(x):= \frac{k!}{|k|!}\sum_{\sigma\in perm(|k|)} z_{j_{\sigma(1)}}z_{j_{\sigma(2)}} \ldots z_{j_{\sigma(|k|)}},
$$
where $k$ represents a multi-index. They
play the role of the sequence of the complex polynomials $(z^n)_n$ when $z$ is a complex variable.
In order to preserve the monogenicity we need a special product, called Cauchy-Kowalewski product, for short CK-product,
that does not coincide with the ordinary pointwise product.
The  CK-product  of two left entire monogenic $f$ and $g$ is denoted by
$f \odot_L g$ is given in Definition \ref{CKPRD} in terms of the Fueter polynomials. A similar definition is given for right monogenic functions.

\medskip
In the paper \cite{ANMATHPHY} we obtained the following result
regarding  monogenic functions.
Let $p\geq 1$ and set $\mathbb{N}_0=\mathbb{N} \cup\{0\}$.
Let $(u_m)_{m\in (\mathbb{N}_0)^n }:\mathbb{R}^{n+1}\to \mathbb{R}_n$ be  left entire monogenic functions
such that for every $\varepsilon>0$ there exist $B_\varepsilon >0$, $C_\varepsilon>0$ for which
\begin{equation}
 |u_m(x)|\leq C_\varepsilon \frac{\varepsilon^{|m|}}{(|m|!)^{1/q}}\exp(B_\varepsilon|x|^p), \ \ \ {\rm for \ all}  \ \ \ m\in (\mathbb{N}_0)^n,
 \end{equation}
where $1/p+1/q=1$ and where we set $1/q:=0$ when $p=1$, and $m$ is a multi-index.
We considered the formal infinite order differential operator
\begin{equation}\label{UL}
U_L(x,\pp_{x})f(x):=\sum_{|m|=0}^\infty u_m(x)\odot_L \pp_{x}^m f(x),
\end{equation}
for left entire monogenic functions $f$
where $\pp_{x}^m:= \pp_{x_1}^{m_1}\dots\pp_{x_n}^{m_n}$ and  $\odot_L$ denotes the CK-product.
Then for  $p\geq 1$, we proved that the operator
$U_L(x,\pp_{x})$
 acts continuously on the space of left monogenic functions with the condition
 $|f(x)|\leq C e^{B|x|^p}.$

\medskip
In this paper we
characterize the continuous homomorphisms of type (\ref{UL}) acting on monogenic functions.
In order to do this we introduce proximate orders
for monogenic functions and we study some fundamental properties. After that we investigate the monogenic counterpart
of the differential operator representation of continuous homomorphisms
between the spaces of entire functions of given proximate order proved
by T. Aoki and co-authors in \cite{Aoki2,AIO20}.

\medskip
{\em The plan of the paper.}
Section \ref{INTROD1} provides an introduction.
In Section \ref{SEC2} we state some preliminary results on monogenic functions.
 In Section \ref{Prel}
we study  entire monogenic  functions where the growth is determined by a proximate order. We establish some important  properties on the proximate order of monogenic functions.
  In Section \ref{SEC3} we then apply these results to characterize the continuous homomorphisms. We observe that Section \ref{Prel} and Section \ref{SEC3} closely follow Section $3$ and Section $4$ of \cite{AIO20}. The points where the adaptation to the case of monogenic functions are not straightforward are Theorem \ref{tsupernovissima1}, Theorem \ref{main1} and Theorem \ref{main2}. In particular, in the last two theorems we have used the new Lemmas \ref{l11} and \ref{l12}.

\section{Preliminary results on monogenic functions}\label{SEC2}

In this section we recall some results on  monogenic functions, whose proofs can be found in
\cite{BDS82}.
We recall that $\rr_n$ is the real Clifford algebra over $n$ imaginary units $e_1,\ldots ,e_n$.
The element
 $(x_0,x_1,\ldots,x_n)\in \mathbb{R}^{n+1}$ will be identified with the paravector
$
 \bfx=x_0+\underline{x}=x_0+ \sum_{\ell=1}^nx_\ell e_\ell
$
and the real part $x_0$ of $\bfx$ will also be denoted by ${\rm Re}(\bfx)$.
An element $y$ in $\mathbb{R}_{n}$, is called a {\em Clifford number}.
If $A$ is an element in the power set $P(1,\ldots ,n)$,
then the element $e_{i_1}\ldots e_{i_r}$ can be written as $e_{i_1...i_r}$ or, in short, $e_A$.
Thus, we can write a Clifford number as
$
y=\sum_Ay_A e_A.
$
Possibly using the defining relations $e_i^2=-1$, $e_ie_j+e_je_i=0$, $i,j\in\{1,\ldots, n\}$, $i\not=j$, we will order the indices in $A$ as $i_1 < \ldots <i_r$. When $A=\emptyset$ we set $e_\emptyset=1$.
The Euclidean norm of an element $y\in \mathbb{R}_n$ is
given by $|y|^2=\sum_{A} |y_A|^2$,
in particular the norm of the paravector $\bfx\in\mathbb{R}^{n+1}$ is $|\bfx|^2=x_0^2+x_1^2+\ldots +x_n^2$.

\begin{definition}
We call $f:\mathbb R^{n+1}\to \mathbb R_n$ an entire left (right) monogenic function if $f\in \mathcal{C}^{1}(\mathbb R^{n+1},\mathbb R_n)$ and $\mathcal{D}f\equiv 0$ (resp. $f\mathcal{D}\equiv 0$) where
$$ \mathcal{D}:=\partial_{x_0}+\sum_{i=1}^n e_i \partial_{x_i}. $$
\end{definition}
We will restrict ourselves to consider only the set of the left monogenic functions since for right monogenic functions analogous computations hold.

We will denote the set of left entire monogenic function by $\mathcal M_L(\mathbb R^{n+1})$.
The Cauchy formula
for monogenic functions and their derivatives, is computed on the boundary
of $U \subset \mathbb{R}^{n+1}$ where $\overline{U}$ is contained in the set of monogenicity of $f$.

\begin{theorem}
 If $f$ is a left monogenic function in a neighborhood of $\overline U$ then for any $\bfx\in U$
$$ f(\bfx)= \frac 1{A_{n+1}}\int_{\partial U} q_0(\bfx-\xi) d\tau(\xi) f(\xi)$$
where $A_{n+1}$ is the $n$-dimensional surface area of the $n+1$-dimensional unit ball,
$$ d\tau(\xi)=\sum_{j=0}^n(-1)^j e_j \widehat{d\xi_j} \quad\textrm{and}\quad q_0(\bfx)=\frac{\bar\bfx}{\|\bfx\|^{n+1}}$$
with
$$\widehat{d\xi_j}=d\xi_0\wedge\dots\wedge d\xi_{j-1}\wedge d\xi_{j+1}\wedge\dots\wedge d\xi_n.$$
\end{theorem}
\begin{theorem}
Let $\bfm:=(m_1,\dots, m_n)\in \mathbb N_0^n$ be a multi-index. We denote by $|\bfm|=m_1+\dots+m_n$. If $f$ is a left monogenic function in a ball $\|\bfx\|< R$, then for all $\|\bfx\|<r$ with $0<r<R$, its Taylor series expansion is given by
$$ f(\bfx)=\sum_{|\bfm|=0}^\infty V_{\bfm}(\bfx) a_{\bfm} $$
where $V_{\bfm}$ are the Fueter polynomials and
$$a_{\bfm}=\frac 1{\bfm ! A_{n+1}}\int_{\|\xi\|\leq r}q_{\bfm}(\xi) d\tau(\xi) f(\xi)$$
where $q_{\bfm}(\bfx)=\partial_{x_1}^{m_1}\dots\partial_{x_n}^{m_n}q_0(\bfx)$. We have the following Cauchy inequality
$$
\|a_{\bfm}\|\leq c(n,\bfm) \frac{M(r,f)}{r^{|\bfm|}},
$$
where
$$
c(n,\bfm):=\frac{n(n+1)\cdot\dots\cdot(n+|\bfm|-1)}{\bfm!}
$$
 and $M(r,f)=\sup_{\|\bfx\|=r} \|f(\bfx)\|$. In particular, setting $\bfm:=(m_0,\dots, m_n)\in\mathbb N^{n+1}_0$, we have
\begin{equation}\label{estimate1}
\|\partial_{x_0}^{m_0}\dots \partial_{x_n}^{m_n}q_0(\bfx)\|\leq \frac{n(n+1)\cdots(n+|\bfm|-1)}{\|\bfx\|^{n+|\bfm|}}.
\end{equation}
\end{theorem}
\begin{definition}\label{CKPRD}
Let $f,\, g\in\mathcal M_L(\mathbb R^{n+1})$ (resp. $f,\, g\in\mathcal M_R(\mathbb R^{n+1})$). Using their Taylor series representation
$$ f(\bfx)=\sum_{|\bfm|=0}^{\infty} V_{\bfm}(\bfx) a_{\bfm},\quad \textrm{(resp. $f(\bfx)=\sum_{|\bfm|=0}^{\infty} a_{\bfm} V_{\bfm}(\bfx) $)} $$
and
$$ g(\bfx)=\sum_{|\bfm|=0}^{\infty} V_{\bfm}(\bfx) b_{\bfm},\quad \textrm{(resp. $g(\bfx)=\sum_{|\bfm|=0}^{\infty} b_{\bfm} V_{\bfm}(\bfx) $)} $$
we define
$$ f\odot_L g:= \sum_{|\bfm|=0}^{\infty}\sum_{|\bfk|=0}^{\infty} V_{\bfm+\bfk}(\bfx) a_\bfm b_\bfk \quad\textrm{$\left( \textrm{resp. } f\odot_R g:= \sum_{|\bfm|=0}^{\infty}\sum_{|\bfk|=0}^{\infty}  a_\bfm b_\bfk V_{\bfm+\bfk}(\bfx) \right)$}.$$
\end{definition}

\medskip
We mention that a different product, that is defined just for the subclass of axially monogenic function
(see \cite{ANUNO} and \cite{ANDUE}), can also be used to define infinite order differential operators in the monogenic setting.

\medskip
Next, following \cite{AlmeidaKra2005}, we recall the definition of the standard growth order of an entire monogenic function:
\begin{definition}
Let $f: \mathbb{R}^{n+1} \to \mathbb{R}_n$ be an entire left monogenic function. Then its growth order is said to be
$$
\rho = \rho(f) = \limsup\limits_{r \to +\infty} \frac{\log^+ \log^+ (M(r,f))}{\log r},
$$
where $\log^+(r) = \max\{0,\log(r)\}$.
\end{definition}
It may occur that $0 \le \rho \le +\infty$. For the case where $0 < \rho < +\infty$ Constales et al. defined in \cite{CDK07bis} the growth type of an entire monogenic function of growth order $\rho$ by
$$
\sigma = \sigma(f) = \limsup\limits_{r \to \infty+} \frac{\log^+ M(r,f)}{r^{\rho}}.
$$
As shown in \cite{CDK07} the growth order of a monogenic function can be directly computed by its Taylor coefficients, namely by
$$
\rho(f) = \limsup\limits_{|{\bf m}| \to + \infty} \frac{|{\bf m}| \log |{\bf m}|}{-\log\Big|\frac{1}{c(n,{\bf m}) }a_{\bf m}  \Big|}.
$$
Similarly, as shown in \cite{CDK07bis} the growth type can be expressed by
$$
\sigma(f) = \frac{1}{e\rho} \limsup\limits_{r \to +\infty}|{\bf m}|\Big(|a_{\bf m}| \Big)^{\frac{\rho}{|{\bf m}|}}.
$$
\begin{remark}
To also get a finer classification of functions with slow growth $\rho=0$ and fast growth $\rho=\infty$ Seremeta \cite{Seremeta70} and Shah \cite{Shah77} introduced the notion of generalized growth in the complex analysis setting which has been generalized to the monogenic setting in \cite{CDK2014,KB2012,KB2013,SK2013}.
More precisely the authors considered functions $\alpha(\cdot)$, $\beta(\cdot)$ and $\gamma(\cdot)$ satisfying particular conditions mentioned concisely for example in Definition~1 of \cite{CDK2014} and introduced the notions of the generalized growth and type in the way
$$
\rho_{\alpha,\beta}(f) = \limsup\limits_{r \to +\infty} \frac{\alpha(\log^+ M(r,f))}{\beta(\log(r))}
$$
and
$$
\sigma_{\alpha,\beta,\gamma}(f) = \limsup\limits_{r \to +\infty} \frac{\alpha(\log^+ M(r,f))}{\beta((\gamma(r))^{\rho_{\alpha,\beta}})},
$$
(see for instance Definition~2 of \cite{CDK2014} or \cite{KB2013}).
In the particular case where $\alpha(r)=\log(r)$ and where $\beta$ is the identity function one re-obtains the classical definition of the growth order and growth type. Another generalization of the classical growth order that includes the classical growth order as a special case, and even the generalized growth orders under particular conditions, is the definition of the proximate growth order.
\end{remark}

We introduce:
\begin{definition}
A differentiable function $\rho(r)\geq 0$ defined for $r\geq 0$ is said to be a proximate order for the order $\rho\geq 0$ if it satisfies
\begin{enumerate}
\item $\lim_{r\to+\infty} \rho(r)=\rho$,
\item $\lim_{r\to+\infty} \rho '(r)r\ln(r)=0.$
\end{enumerate}
\end{definition}
We observe that for any proximate order function $\rho(r)$ there exists a positive constant $r_0>0$ such that for any $r>r_0$ the function $r^{\rho(r)}$ is strictly increasing and tending to $+\infty$.
\begin{definition}\label{norm}
For any proximate order function $\rho(r)$ we can always take another proximate order function, called the normalization of the proximate order function $\rho(r)$, $\hat{\rho}(r)$ such that there exists a constant $r_1>0$ for which  $\hat{\rho}(r)=\rho(r)$ for any $r\geq r_1$ and $r^{\hat{\rho}(r)}$ is strictly increasing on $r>0$ and maps the interval $(0,+\infty)$ to $(0,+\infty)$.
\end{definition}
We denote by $\varphi:(0,+\infty)\to(0+\infty)$ the inverse function of the function $t=r^{\hat{\rho}(r)}$. Moreover, we set
\begin{equation}\label{c1}
G_q=G_{\hat{\rho},q}:=\frac{\varphi(q)^q}{(e\rho)^{q/\rho}},\quad\textrm{for $q\in\mathbb N$.}
\end{equation}
Let $\rho(r)$ be a proximate order for a positive order $\rho>0$. For any $\sigma>0$, we consider the Banach space
$$ A_{\rho,\sigma}:=\{f\in\mathcal M_L(\mathbb R^{n+1}):\, \|f\|_{\rho,\sigma}:=\sup_{\mathbf{x}\in\mathbb R^{n+1}} \|f(\bfx)\|\exp(-\sigma\|\bfx\|^{\rho(\|\bfx\|)})<+\infty\} $$
with the norm $\|\cdot\|_{\rho,\sigma}$. Moreover, we use the following notation: for any $\bfm=(m_1,\dots, m_n)\in\mathbb N^n$ we write
$$\partial^\bfm_\bfx f(\bfx):=\partial_{x_1}^{m_1}\dots\partial_{x_n}^{m_n} f(\bfx).$$

We recall some basic properties of proximate orders that will be used in the following.
 The proofs of these results can be found in Section 2 of \cite{AIO20} and in the references therein.
\begin{lemma}
There exist constants $k>0$ and $B>0$ depending only on $\hat{\rho}$ such that
$$ \forall r>0,\, \forall s>0,\quad (r+s)^{\hat{\rho}(r+s)}\leq k(r^{\hat{\rho}(r)}+s^{\hat{\rho}(s)})+B. $$
Precisely speaking, we can choose $k$ depending only on the order $\rho=\lim_{r\to+\infty} \hat{\rho}(r)$
\end{lemma}
\begin{lemma}\label{lsupersupernova2}
The sequence $\{G_p\}_p$ is supermultiplicative, that is,
$$G_pG_q\leq G_{p+q},\quad \textrm{for any $p,\, q\in\mathbb N$.} $$
\end{lemma}
\begin{lemma}
For every $\delta>0$ with $\delta<\frac 1\rho$, there exists $T_0>0$ such that if $t\geq T_0$, we have
$$\left(\frac 1\rho-\delta\right)\frac d{dt} \ln(t)<\frac d{dt}\ln \varphi(t)<\left(\frac 1\rho+\delta\right)\frac d{dt} \ln(t).$$
\end{lemma}
\begin{lemma}\label{l24}
For $u,\, t,\, \sigma>0$ we define
$$ y_\sigma(u,t):=\ln\frac{\varphi(t)}{\varphi(u)}-\sigma\frac tu. $$
Then for any $\sigma'$ with $0<\sigma'<\sigma$, there exists $T_1$ such that
$$y_\sigma(u,t)+\frac 1\rho\ln(e\rho)\leq -\frac 1\rho\ln(\sigma'),\quad\textrm{for any $u,\, t\geq T_1$.}$$
\end{lemma}

Keeping in mind the above results we can now introduce the notion of  monogenic functions of proximate order and we can study some properties.

\section{Some properties of monogenic functions of proximate order}\label{Prel}

In the following we will use some results on
monogenic entire functions contained in \cite{CDK07,CDK07bis,CK02}.
We prove some new properties of entire slice monogenic functions that appear here for the first time to the best of the knowledged of the authors.
Some of the difficulties in proving our results relay
on in the series expansion of these functions in terms of the Fueter polynomials.
\begin{lemma}\label{lsupernova1}
If $\sigma_2>\sigma_1>0$, then the inclusion map $A_{\rho,\sigma_1}\xhookrightarrow{} A_{\rho,\sigma_2}$ is compact.
\end{lemma}
\begin{proof}
We will show that $B:=\{f\in A_{\rho,\sigma_1}:\, \|f\|_{\rho,\sigma_1}\leq 1\}$ is relatively compact in $A_{\rho,\sigma_2}$, i.e., any sequence $\{f_j\}_{j\in\mathbb N}\subset B$ has an accumulation point with respect to the norm of $A_{\rho,\sigma_2}$.

First we will prove that any sequence $\{f_j\}_{j\in\mathbb N}\subset B$ admits a convergent subsequence in the uniform convergence topology to an entire monogenic function. By the Arzel\'a-Ascoli Theorem it is sufficient to prove that $\{f_j\}_{j\in\mathbb N}$ is equicontinuous and uniformly bounded in any compact convex set $K\subseteq\mathbb R^{n+1}$. We fix a compact convex subset $K$ of $\mathbb R^{n+1}$. Since $\{f_j\}_{j\in\mathbb N}\subset B$, the sequence is uniformly bounded int $K$. Moreover, we have
$$ \|f_j(\bfx)-f_j(\bfy)\|\leq C_j\|\bfx-\bfy\|,$$
where $\nabla$ is the usual gradient, $C_j=\sup_{\bfx\in K}\|\nabla f_j(\bfx)\|$ and $\bfx,\, \bfy\in K$. We choose $r$ large enough in a such way that $K\subset B(0,r)$. Thus, there exists a positive constant $C_K$ which only depends on $K$ such that for any $\bfx\in K$ and for any $j\in\mathbb N$, we have
\[
\begin{split}
\|\nabla f_j(\bfx)\|&=\|\frac 1{A_{n+1}}\int_{\partial B(0,r)}\nabla q_0(\bfx-\xi)d\tau(\xi) f_j(\xi)\|\\
&\leq C_1 \int_{\partial B(0,r)} \frac {1}{\|\bfx-\xi\|^{1+n}} |d\tau(\xi)| M(r,f_j)\leq C_2 \frac{M(r,f_j)}{\operatorname{dist}(\bfx,\partial B(0,r))^{1+n}}\leq C_K.
\end{split}
\]
where $C_1$ and $C_2$ are suitable positive constants while in the first inequality we have used \eqref{estimate1}.
In particular, for any $\bfx,\, \bfy\in K$ and for any $j\in\mathbb N$ we have
$$
\|f_j(\bfx)-f_j(\bfy)\|\leq C_K\|\bfx-\bfy\|,
$$
 i.e., $\{f_j\}_{j\in\mathbb N}$ is equicontinuous. After taking a subsequence if necessary, we can suppose that the sequence $\{f_j\}_{j\in\mathbb N}$ converges to $f\in\mathcal M_L(\mathbb R^{n+1})$ in the topology of the uniform convergence. Now we will prove that the sequence $\{f_j\}_{j\in\mathbb N}$ is a Cauchy sequence in $A_{\rho,\sigma_2}$. We fix $\delta>0$ and we observe that for any $R>0$ we have
\[
\begin{split}
& \sup_{\bfx\in \mathbb R^{n+1}}\|f_j(\bfx)-f_\ell(\bfx)\|\exp(-\sigma_2\|\bfx\|^{\rho(\|\bfx\|)})\\
& =\max\left\{\sup_{\|\bfx\|\leq R}\|f_j(\bfx)-f_\ell(\bfx)\|\exp(-\sigma_2\|\bfx\|^{\rho(\|\bfx\|)}), \sup_{\|\bfx\|\geq R}\|f_j(\bfx)-f_\ell(\bfx)\|\exp(-\sigma_2\|\bfx\|^{\rho(\|\bfx\|)})\right\}.
\end{split}
\]
With respect to the second supremum, we can choose $R>0$ large enough such that:
$$
\exp ((\sigma_1-\sigma_2)\|\bfx\|^{\rho(\|\bfx\|)})\leq \frac \delta2, \ \  \  {\rm for\ any}\  \|\bfx\|\geq R.
$$
 Thus, since $f_j,\, f_\ell\in B$, we have that
\[
\begin{split}
&\sup_{\|\bfx\|\geq R}\|f_j(\bfx)-f_\ell(\bfx)\|\exp(-\sigma_2\|\bfx\|^{\rho(\|\bfx\|)})\\
&=\sup_{\|\bfx\|\geq R}\|f_j(\bfx)-f_\ell(\bfx)\|\exp(-\sigma_1\|\bfx\|^{\rho(\|\bfx\|)}) \exp((\sigma_1-\sigma_2)\|\bfx\|^{\rho(\|\bfx\|)})\leq 2\cdot\frac \delta 2=\delta.
\end{split}
\]
Moreover, with respect to the first supremum, by the uniform convergence of the sequence $\{f_j\}_{j\in\mathbb N}$ on the compact subset of $\mathbb R^{n+1}$, there exists a positive integer $N$ such that for any $j,\, \ell\geq N$ we have
$$ \sup_{\|\bfx\|\leq R}\|f_j(\bfx)-f_\ell(\bfx)\|\exp(-\sigma_2\|\bfx\|^{\rho(\|\bfx\|)})\leq \sup_{\|\bfx\|\leq R}\|f_j(\bfx)-f_\ell(\bfx)\|\leq \delta. $$
Thus we have proved that the sequence $\{f_j\}_{j\in\mathbb N}$ is a Cauchy sequence.
\end{proof}
\begin{definition}[The spaces $A_\rho$ and $A_{\rho,\sigma+0}$] We define the space
$$A_\rho:=\lim_{\underset{\sigma>0}{\rightarrow}} A_{\rho,\sigma} $$
i.e. $A_{\rho}=\cup_{\sigma>0}A_{\rho,\sigma}$ and we say that a sequence $\{f_j\}_{j\in\mathbb N}\subseteq A_{\rho}$ converges to $f\in A_\rho$ if there exists $\sigma>0$ such that $\{f_j\}_{j\in\mathbb N}\subseteq A_{\rho,\sigma}$, $f\in A_{\rho,\sigma}$ and $\lim_{j\to+\infty} \|f_j-f\|_{\rho,\sigma}=0$.

We also define the space
$$A_{\rho,\sigma+0}:=\lim_{\underset{\epsilon>0}{\leftarrow}} A_{\rho, \sigma+\epsilon}$$
i.e., $A_{\rho,\sigma+0}:=\cap_{\epsilon>0} A_{\rho,\sigma+\epsilon}$ and we say that a sequence $\{f_j\}_{j\in\mathbb N}\subseteq A_{\rho,\sigma+0}$ converges to $f\in A_{\rho,\sigma+0}$ if for any $\epsilon>0$ we have $\{f_j\}_{j\in\mathbb N}\subseteq A_{\rho,\sigma+\epsilon}$, $f\in A_{\rho,\sigma+\epsilon}$ and $\lim_{j\to+\infty} \|f_j-f\|_{\rho,\sigma+\epsilon}=0$.
\end{definition}
\begin{remark}
The spaces $A_\rho$ and $A_{\hat{\rho}}$ coincide with each other and they share the same locally convex topologies as well, and the same holds for $A_{\rho,\sigma+0}$ and $A_{\hat{\rho},\sigma+0}$.
\end{remark}
For any fixed $f(\bfx)=\sum_{|\bfm|=0}^{\infty} V_\bfm(\bfx) a_\bfm\in\mathcal M_L(\mathbb R^{n+1})$ we define
$$ K_q:=\sup_{\|\bfx\|\leq 1}\|\sum_{|\bfm|=q} V_\bfm(\bfx)a_\bfm\| $$
and
$$ P_q(\vx):=\sum_{|\vm|=q} V_\vm(\vx)a_\vm. $$
\begin{theorem}\label{tsupernovissima1}
Let $f(\bfx)=\sum_{|\bfm|=0}^{\infty} V_\bfm(\bfx) a_\bfm\in\mathcal M_L(\mathbb R^{n+1})$ be a left monogenic entire function of finite order $\rho>0$ and of proximate order $\rho(r)$. Then its type $\sigma$ with respect to $\rho(r)$ is given by
\begin{equation}\label{kq}
\frac 1\rho\ln(\sigma)=\limsup_{q\to\infty}\left(\frac 1q\ln K_q+\ln\varphi(q)\right)-\frac 1\rho-\frac{\ln(\rho)}{\rho}.
\end{equation}
\end{theorem}
\begin{proof}
First we prove that
$$
\frac 1\rho\ln(\sigma)\geq\limsup_{q\to\infty}\left(\frac 1q\ln K_q+\ln\varphi(q)\right)-\frac 1\rho-\frac{\ln(\rho)}{\rho}.
$$
Let
$$
P_q(\bfx)=\sum_{|\bfm|=q} V_{\bfm}(\bfx)a_{\bfm}
$$
and suppose that $\bfw_q\in\mathbb R^{n+1}$ be such that $\|P_q(\bfw_q)\|=K_q$ and $\|\bfw_q\|=1$. There exists $\bfb\in \mathbb R_{n}$ (see \cite[Theorem $3.20$]{GHS08}) such that
\begin{itemize}
\item $\bfb=\bfz_1\cdots \bfz_r$ where $\bfz_i\in\mathbb R^{n+1}$ and $\|\bfz_i\|=1$ for any $i=1,\dots r$,\\
\item $\bfw_q:= \bfb (x\bfe_1) \bar \bfb$ for some $x\in\mathbb R$ with $|x|\leq 1$ where $\bar \bfb:= \bar \bfz_r\cdots \bar \bfz_1 $.
\end{itemize}
The function $g(\bfx):=\bfb f(\bfb \bfx \bar\bfb)$ is a left monogenic entire function (see \cite{R85}) whose Taylor series is
$$
g(\bfx)=\sum_{|\bfm|=0}^{\infty} V_\bfm(\bfx) a_\bfm '.
$$
 We define
 $$
 \tilde P_q(\bfx):=\sum_{|\bfm|=q} V_{\bfm}(\bfx)a_\bfm'.
 $$
  Thus we have
$$K_q=\|P_q(\bfw_q)\|=\|\tilde P_q(x\bfe_1)\|=\|x^q a_{\bfm_q}'\|\leq \|a_{\bfm_q}'\|\leq \frac{c(n,\bfm_q) M(r,g)}{r^q}= \frac{c(n,\bfm_q) M(r,f)}{r^q},$$
where $\bfm_q=(q,0,\dots,0)$. If $\tilde \sigma>\sigma$ then for $r$ large we have:
$$M(r,f)\leq \exp\left(\tilde \sigma r^{\rho(r)}\right),$$ and
\begin{equation}\label{kq1}
\ln(K_q)\leq \ln(c(n,\bfm_q))+\tilde\sigma r^{\rho(r)}-q\ln(r).
\end{equation}
If $q$ is large enough, then we define $r_q$ to be the real number such that $q=\tilde\sigma \cdot\rho\cdot (r_q)^{\rho(r_q)}$ and we have $\varphi\left(\frac{q}{\tilde\sigma \rho}\right)=r_q$. Thus, for $q$ large enough, dividing by $q$ and summing $\log(\varphi(q))$ to both side of inequality \eqref{kq1}, we have
\begin{equation}\label{kq2}
\ln\left(\varphi(q)K_q^{\frac 1q}\right)<\ln\left(c(n,\bfm_q)^{\frac 1q}\right)+\frac 1\rho+\ln\left(\frac{\varphi(q)}{\varphi\left(\frac q{\tilde \sigma \rho}\right)}\right).
\end{equation}
By \cite[(1) Theorem 1.23]{LG86} we have
$$
\lim_{q\to+\infty}\frac{\varphi(q)}{\varphi\left(\frac q{\tilde\sigma\rho}\right)}=(\tilde\sigma \rho)^{\frac 1\rho}.
$$
 Moreover, since  $\lim_{q\to +\infty} \left(c(n,\bfm_q)\right)^{\frac 1q}=1$, taking the $\limsup$ to both side of \eqref{kq2}, we have
$$\limsup_{q\to+\infty} \ln\left( \varphi(q) K_q^{\frac 1q} \right) \leq \ln\left((e\tilde \sigma \rho)^{\frac 1\rho}\right)$$
for any $\tilde\sigma>\sigma$.
Thus we have proved that
 $$\limsup_{q\to+\infty} \ln\left(\varphi(q) K_q^{\frac 1q} \right) \leq \ln\left( (e\sigma \rho)^{\frac 1\rho}\right).
  $$
  The other side of the inequality follows as in \cite[(3) Theorem 1.23]{LG86} by the properties of $\varphi$.
\end{proof}

\begin{lemma}\label{lsupernovissima1}
A left monogenic entire function $f(\bfx)=\sum_{|\bfm|=0}^{+\infty}V_\bfm(\bfx) a_\bfm\in\mathcal M_L(\mathbb R^{n+1})$ belongs to $A_{\rho,\sigma+0}$ if and only if we have
$$\limsup_{q\to\infty}(K_qG_{\hat{\rho},q})^{\frac \rho q}\leq \sigma.$$
\end{lemma}
\begin{proof}
We have that $f\in A_{\rho,\sigma+0}$ if and only if for any $\epsilon>0$ there exists a $D_\epsilon>0$ such that
$$ \|f(\vx)\|\leq D_\epsilon \exp((\sigma+\epsilon) \|\vx\|^{\rho(\vx)}) \quad\textrm{for all $\vx\in\mathbb R^{n+1}$}. $$
This is equivalent to
 \begin{equation}\label{newnova1}
 \limsup_{r\to+\infty} \frac{\sup_{\|\vx\|\leq r} \ln(\| f(\vx)\|)}{r^{\rho (r)}} \leq \sigma.
 \end{equation}
Since
$$
\limsup_{r\to+\infty} \frac{\sup_{\|\vx\|\leq r} \ln(\| f(\vx)\|)}{r^{\rho (r)}}
$$
 is the type of $f$ with respect to $\rho(r)$, by Theorem \ref{tsupernovissima1} and inequality \eqref{newnova1} we have
$$
\limsup_{q\to+\infty} \left( \frac \rho q \ln(K_q)+\rho\ln(\varphi(q)) \right)-\ln(e\rho)\leq\ln(\sigma).
$$
Finally, we can conclude that
$$
\limsup_{q\to+\infty}\, \ln\left( K_q^{\rho/q}\frac{\varphi(q)^\rho}{e\rho} \right)\leq \ln(\sigma)
$$
which gives
$$
\limsup_{q\to+\infty}\left( K_q \frac{\varphi(q)^q}{(e\rho)^{q/\rho}} \right)^{\rho/q}=\limsup_{q\to+\infty}\left( K_q G_{\hat{\rho},q} \right)^{\rho/q}\leq \sigma.
$$
\end{proof}
\begin{corollary}\label{csupersupernovissima2}
If $f(\bfx)=\sum_{|\bfm|=q}V_\bfm(\bfx) a_\bfm\in\mathcal M_L(\mathbb R^{n+1})$ belongs to $f\in A_{\rho, \sigma+0}$, then we have
\begin{equation}\label{nova}
\limsup_{q\to\infty}\left(\max_{|\bfm|=q } \|a_\bfm\|G_{\hat{\rho},q}\right)^{\rho/q}\leq n^\rho\sigma.
\end{equation}
Conversely, if $\{a_\bfm\}$ satisfies \eqref{nova}, then we have $f(\bfx)\in A_{\rho, n^\rho\sigma+0}$.
\end{corollary}
\begin{proof}
We have
$$
\|a_\vm\|=\frac 1{\vm!}\|\partial_\vx^\vm P_q(0)\|\leq \frac {c(n,\vm) \sup_{\|x\|\leq r} \|P_q(\vx)\|}{r^q}=c(n,\vm) \sup_{\|x\|\leq 1} \|P_q(\vx)\|=c(n,\vm) K_q.
$$
Combining the previous estimate with Lemma \ref{lsupernovissima1} we get that for any $\epsilon>0$ there exists $q_0>0$ such that for any $q>q_0$ we have
\[
\begin{split}
\left( \max_{|\vm|=q} \|a_\vm\| G_{\rho,q} \right)^{\rho/q}=\left( \max_{|\vm|=q} c(n,\vm) K_q G_{\rho, q} \right)^{\rho/q}& \leq \left( \left (\sum_{|\vm|=q} c(n,\vm)\right)^{1/q}\right)^\rho (K_q G_{\rho, q})^{\rho/q} \\
& \leq (n+\epsilon)^\rho(\sigma+\epsilon),
\end{split}
\]
which implies
$$
\limsup_{q\to+\infty} \, \left( \max_{|\vm|=q} \|a_\vm\| G_{\rho, q} \right)^{\rho/q}\leq n^\rho\sigma.
$$
Conversely, if we have \eqref{nova}, then for any $\epsilon>0$ there exists $C>0$ and $q_0>0$ such that for any $q>q_0$
we have
\[
\begin{split}
K_q=\sup_{\|\bfx\|\leq 1} \|P_q(\vx)\|\leq \sup_{\|\bfx\|\leq 1} \sum_{\|\vm\|=q} \|a_\vm\| \|\vx\|^q
\leq \max_{|\vm |=q} \|a_\vm\| G_{\rho,q} \frac{(q+1)^{n-1}}{G_{\rho,q}}
\leq (n^\rho\sigma+\epsilon)^{q/\rho} \frac{(q+1)^{n-1}}{G_{\rho,q}}.
\end{split}
\]
Thus, we have
$$
\limsup_{q\to +\infty}(K_q G_{\rho,q})^{\rho/q}\leq n^\rho\sigma
$$
and $f\in A_{\rho,n^\rho\sigma+0}$.
\end{proof}
The following proposition is a direct consequence of the previous corollary.
\begin{proposition}
A left monogenic entire function $f(\bfx)=\sum_{|\bfm|=q}V_\bfm(\bfx) a_\bfm\in\mathcal M_L(\mathbb R^{n+1})$ belongs to $A_\rho$ if and only if we have
$$\limsup_{q\to\infty}\left(\max_{|\bfm|=q } \|a_\bfm\|G_{\hat{\rho},q}\right)^{\rho/q}< \infty.$$
\end{proposition}
We now need some estimates of norms of Fueter polynomials.
\begin{lemma}\label{lsupersupernova1}
Suppose $\rho>0$. For any $\sigma$ and $\sigma'$ with $0<\sigma'<\sigma$, there exists $C>0$ such that for any $\bfm\in\mathbb N^n_0$, we have
$$\|V_\bfm(\bfx)\|_{\rho,\sigma}\leq C\sigma'^{-|\bfm|/\rho}G_{\hat{\rho},|\bfm|}.$$
\end{lemma}
\begin{proof} Since the norms $\|\cdot\|_{\rho,\sigma}$ and $\|\cdot\|_{\hat\rho,\sigma}$ are equivalent, we may assume from the beginning that $\hat\rho(r)$ is as in Definition \ref{norm} and $\hat\rho(r)=\rho(r)$.

First we prove, for sufficiently large $r=\|\vx\|$ and $q=|\vm|$, that
$$
\|V_\vm(\vx) e^{-\sigma r^{\rho(r)}} \|\leq (\sigma')^{-|\vm|/\rho} G_{\hat{\rho},|\vm|}.
$$
Let $\varphi$ be the inverse function of $t=r^{\rho(r)}$. We define $r=\varphi(t)=\|\vx\|$. Thus we have
\[
\begin{split}
\|V_\vm(\vx) e^{-\sigma r^{\rho(r)}}\| G^{-1}_{\hat\rho, q} &\leq r^qe^{-\sigma r^{\rho(r)}} G^{-1}_{\hat\rho, q}\\
&=\exp\left (q\left (\ln(\varphi(t))-\sigma\frac tq-\frac 1q\ln(G_{\hat\rho,q})\right)\right)\\
&=\exp\left (q\left (\ln(\varphi(t))-\sigma\frac tq-\ln(\varphi(q))+\frac 1\rho\ln(e\rho)\right)\right)\\
&=\exp\left (q\left (\ln\left (\frac{\varphi(t)}{\varphi(q)} \right)-\sigma\frac tq+\frac 1\rho\ln(e\rho)\right)\right)\\
\end{split}
\]
By Lemma \ref{l24}, for any $0<\sigma'<\sigma$ there exists $T_1\geq 0$ such that
$$ \ln\left (\frac{\varphi(t)}{\varphi(q)} \right)-\sigma\frac tq+\frac 1\rho\ln(e\rho)\leq -\frac 1\rho\ln(\sigma')\quad\textrm{for any $q,t\geq T_1$}. $$
This implies
\begin{equation}\label{espuernovissima6}
\|V_\vm(\vx) e^{-\sigma r^{\rho(r)}}\| G^{-1}_{\hat\rho,q} \leq \exp \left (-\frac q\rho\ln\sigma' \right)=(\sigma')^{-q/\rho}
\end{equation}
for $|\vm|=q\geq T_1$ and $t\geq T_1$ (i.e. $\|\vx\|\geq \varphi(T_1)$). Next we consider the case $\|\vx\|\leq \varphi(T_1)$ (i.e. $t\leq T_1$) and $q=|\vm|>>1$. We have
\[
\begin{split}
\|V_\vm(\vx) e^{-\sigma r^{\rho(r)}}\| G^{-1}_{\hat\rho, q} &\leq r^qe^{-\sigma r^{\rho(r)}} G^{-1}_{\hat\rho, q}\\
&=\exp\left (q\left (\ln\left (\frac{\varphi(t)}{\varphi(q)} \right)-\sigma\frac tq+\frac 1\rho\ln(e\rho)\right)\right)\\
&\leq \exp\left( q\left(\ln \left( \frac {\varphi(T_1)}{\varphi(q)} \right) +\frac 1{\rho} \ln(e\rho)\right) \right).
\end{split}
\]
Note that for any given $\sigma'>0$, we can take $T_2\geq T_1$:
$$\ln\left( \frac{\varphi(T_1)}{\varphi(T_2)} \right) +\frac 1{\rho}\ln(e\rho)\leq -\frac 1\rho\ln(\sigma').$$
Thus we have
\begin{equation}\label{esupernovissima7}
\|V_\vm(\vx) e^{-\sigma r^{\rho(r)}}\| G^{-1}_{\hat\rho,q} \leq (\sigma')^{-q/\rho}
\end{equation}
for $|\vm|=q\geq T_2$ and $\|\vx\|\leq \varphi(T_1)$. By \eqref{espuernovissima6} and \eqref{esupernovissima7} for any $\vx\in\mathbb R^{n+1}$ and for any $|\vm|=q\geq T_2$, we have
$$ \|V_\vm(\vx) e^{-\sigma r^{\rho(r)}}\| G^{-1}_{\hat\rho,q} \leq (\sigma')^{-q/\rho}.$$
We know that there exists $C>0$ such that $\|V_\vm(\vx) e^{-\sigma r^{\rho(r)}}\|\leq C$ for any $|\vm|<T_2$ and for any $\vx\in\mathbb R^{n+1}$. Thus
$$ \|V_\vm(\vx) e^{-\sigma r^{\rho(r)}}\| G^{-1}_{\hat\rho,q} \leq C' (\sigma')^{-q/\rho} $$
where
$$
 C'=\sup_{|\vm|\leq T_2}\left(1, \frac C{(\sigma')^{-q/\rho} G_{\hat \rho, q}}\right).
 $$
\end{proof}

\begin{lemma}\label{lnova1}
There exists a constant $k$ depending only on $\rho$ for which the following statement holds: For any $\sigma>0$, we can take $C(\sigma)$ such that for any $f\in A_{\hat{\rho},\sigma}$, and any $\bfm\in\mathbb N_0^n$, the inequality
$$ \frac 1{\bfm!} \|\partial^\bfm_{\bfx} f(\bfx)\|_{\hat{\rho}, k\sigma}\leq C(\sigma)\|f\|_{\hat{\rho},\sigma}\frac{(2k(c(n,\vm))^{\rho/q}\sigma)^{q/\rho}}{G_{\hat{\rho},q}}$$
holds. Here we write $q=|\bfm|$.
\end{lemma}
\begin{proof}
There exists a constant $k>0$ which depends on $\rho$ and there exists a constant $B>0$ which depends on $\hat\rho(r)$ such that
$$
(r+s)^{\hat \rho(r+s)}\leq k(r^{\hat\rho(r)}+s^{\hat\rho(s)})+B\quad\textrm{for all $r,s>0.$}
$$
The Cauchy estimates give us, for $\|\vx\|\leq r$ and $|\vm|=q$ the chain of inequalities
\[
\begin{split}
\frac{\|\partial^\vm_\vx f(\vx)\|}{\vm!} & \leq
\inf_{s>0} \frac{c(n,\vm) \max_{|\xi|=s} \|f(\vx+\xi)\|}{s^{|\vm|}}
\\
&
\leq \|f\|_{\hat\rho,\sigma} \inf_{s>0} \frac{c(n,\vm)}{s^q} \exp\left( \sigma(r+s)^{\hat\rho(r+s)} \right)
\\
&
\leq \|f\|_{\hat\rho,\sigma} \inf_{s>0} \frac{c(n,\vm)}{s^q} \exp\left( k\sigma(r^{\hat\rho(r)}+s^{\hat\rho(s)}) + B\sigma \right)
\\
&
= e^{B\sigma}\|f\|_{\hat\rho,\sigma} c(n,\vm)  \exp\left( k\sigma r^{\hat\rho(r)}\right) \inf_{s>0} \frac{ \exp\left( k\sigma s^{\hat\rho(s)}\right)}{s^q}
\\
&
\leq e^{B\sigma}\|f\|_{\hat\rho,\sigma}   \exp\left( k\sigma r^{\hat\rho(r)}\right) \frac{\left( c(n,\vm)^{\rho/q} e\right)^{q/\rho}}{\varphi(q)^q} \left( \frac{\varphi(q)}{\varphi(q/(k\sigma\rho))} \right)^q
\\
&
\leq C(\sigma)\|f\|_{\hat\rho,\sigma} \exp\left( k\sigma r^{\hat\rho(r)}\right) \frac{\left( 2k c(n,\vm)^{\rho/q} \sigma\right)^{q/\rho}}{G_{\hat\rho,q}},
\end{split}
\]
where the last two inequalities are obtained as in \cite[Lemma 3.8]{AIO20}. In particular we have
$$ \frac{\|\partial_\bfx^\bfm f(\bfx)\|}{\bfm!} \exp\left( k\sigma \|\bfx\|^{\hat{\rho}(\|\bfx\|)} \right) \leq  C(\sigma)\|f\|_{\hat\rho,\sigma}  \frac{\left( 2k c(n,\vm)^{\rho/q} \sigma\right)^{q/\rho}}{G_{\hat\rho,q}},$$
which implies the inequality in the statement.
\end{proof}
\begin{remark}\label{rsupersupernovissima1}
Observe that since
$$\limsup_{q\to+\infty} \left( \sum_{|\vm|=q} c(n,\vm) \right)^{1/q}=n$$
there exists a constant $C(n)$ which depends only on $n$ such that for any $\vm\in\mathbb N^n$ we have
$$
(c(n,\vm))^{1/|\vm|}\leq C(n).
$$
Thus, in Lemma \ref{lnova1} the estimate can be rewritten in the following way
$$ \frac 1{\bfm!} \|\partial^\bfm_{\bfx} f(\bfx)\|_{\hat{\rho}, k\sigma}\leq C(\sigma)\|f\|_{\hat{\rho},\sigma}\frac{(2k(C(n))^{\rho}\sigma)^{q/\rho}}{G_{\hat{\rho},q}}.$$
\end{remark}
In view of the above stated properties, we can now  prove the following crucial results.
\begin{proposition}\label{psupersupernova1}
For an entire left monogenic function $f(\bfx)$ belonging to $A_{\rho,\sigma+0}$, its Taylor expansion $\sum_{\bfm\in\mathbb N^n}V_\bfm(\bfx) a_\bfm$ converges to $f(\bfx)$ in the space $A_{\rho,n^\rho\sigma+0}$.
In particular, the set of Fueter polynomials is dense in $A_{\rho,+0}$ and also dense in $A_\rho$.
\end{proposition}
\begin{proof}
For the former statement, it suffices to show that
$$ \sum_{\vm\in\mathbb N^n} \|V_\vm(\vx) a_\vm \|_{\rho,n^\rho(\sigma+\epsilon)} $$
is finite for any $\epsilon>0$. By Lemma \ref{lsupersupernova1} we have
$$
\|V_\vm(\vx)\|_{\rho,n^\rho(\sigma+\epsilon)}\leq C_0(n^\rho(\sigma+\epsilon/2))^{-|\vm|/\rho} G_{\hat\rho,|\vm|}.
$$
On the other hand, by Corollary \ref{csupersupernovissima2}, we have
$$\max_{|\vm|=q} \|a_\vm\| G_{\hat\rho,q}\leq C_1(n^\rho(\sigma+\epsilon/4))^{q/\rho}.$$
Therefore, we have
\[
\begin{split}
\sum_{\vm\in\mathbb N^n} \|V_\vm(\vx)a_\vm\|_{\rho, n^\rho(\sigma+\epsilon)}&=\sum_{q\in\mathbb N}\sum_{|\vm|=q} \|a_\vm\| G_{\hat\rho,q} \|V_\vm(\vx)\|_{\rho,n^{\rho}(\sigma+\epsilon)} G^{-1}_{\hat\rho,q}\\
&\leq \sum_{q\in\mathbb N}\sum_{|\vm|=q} C_0 (n^\rho(\sigma+\epsilon/2))^{-q/\rho} C_1 (n^\rho(\sigma+\epsilon/4))^{q/\rho}\\
&\leq C_0C_1 \sum_{q\in\mathbb N} (q+1)^{n-1} \left( \frac{\sigma+\epsilon/4}{\sigma+\epsilon/2} \right)^{q/\rho},
\end{split}
\]
where the last series is convergent. For the latter statement in the case $f\in A_{\rho,+0}$, it follows from the former one with $\sigma=0$ that
$$
\lim_{q\to+\infty} \sum_{|\vm|\leq q} V_\vm(\vx) a_\vm=f(\vx)
$$
in the space $A_{\rho,+0}$. In the case $f\in A_\rho$, there exists $\sigma>0$ such that $f\in A_{\rho,\sigma+0}$. Then the same convergence holds in the space $A_{\rho,n^\rho\sigma+0}$ and therefore also in the space $A_\rho$.
\end{proof}
\begin{lemma}\label{l11}
Let $f\in A_{\rho,\sigma}$ and $f=\sum_{|\bf m|=0}^\infty V_{{\bf m}}({\bf x}) f_{{\bf m}}$. Let $s$ be a real positive number. Then, for any $\eta>0$ there exists $C_\eta>0$ such that for any ${\bf m}\in \mathbb N^n_0$, we have
$$ \|f_{{\bf m}}\|\exp(-\sigma(1+\eta)(s+1)^\rho r^{\rho(r)})\leq \exp(\sigma C_\eta) \frac{\|f\|_{\rho,\sigma} c(n,{\bf m})}{(sr)^{|{\bf m}|}}. $$
\end{lemma}
\begin{proof}
We have that
\[
\begin{split}
\|f_{{\bf m}}\|& =\frac{\|\partial^{\bf m}f (0)\|}{{\bf m}!}\leq \sup_{\|{\bf x}\|=r}\frac{\|\partial^{\bf m}f({\bf x})\|}{{\bf m}!}\leq \sup_{\|{\bf x}\|=r}\frac{ (\sup_{\|\zeta-{\bf x}\|=s\|{\bf x}\|} \|f(\zeta)\|) c(n,{\bf m})}{(sr)^{|{\bf m}|}}\\
&\leq \exp(\sigma((s+1)r)^{\rho((s+1)r)})  \frac{ \left( \sup_{\|{\bf x}\| =(s+1)r} \|f({\bf x})\| \exp(-\sigma((s+1)r)^{\rho((s+1)r)})\right) c(n,{\bf m})}{(sr)^{|{\bf m}|}}\\
&\leq \exp(\sigma(1+\eta)(s+1)^\rho r^{\rho(r)}+\sigma C_\eta) \frac{\|f\|_{\rho,\sigma} c(n,{\bf m})}{(sr)^{|{\bf m}|}},
\end{split}
\]
where in the first inequality we have used the maximum modulus principle, in the second inequality we have used the Cauchy inequalities in the ball centered at ${\bf x}$ with radius $s\|{\bf x}\|$, in the third inequality we have used the fact that all the balls centered at ${\bf x}$ with $\|{\bf x}\|=r$ of radius $s\|{\bf x}\|$ are contained in the ball centered at $0$ with radius $(s+1)r$ and in the last inequality we have used the fact that for any $\eta>0$ there exists a positive constant $C_\eta$ such that
$$ (kr)^{\rho(kr)} \leq (1+\eta) k^\rho r^{\rho(x)} + C_\eta, $$
see \cite[p. $16$, Proposition $1.20$]{LG86}.
\end{proof}

\begin{lemma}\label{l12}
Let $g_1({\bf x})=\sum_{|{\bf m}|=0}^{\infty} V_{\bf m}({\bf x})a_{{\bf m}}\in A_{\rho, \tau_1}$ and $g_2({\bf x})=\sum_{|{\bf m}|=0}^{\infty} V_{\bf m}({\bf x})b_{{\bf m}}\in A_{\rho, \tau_2}$. Let $\delta$ be a positive constant then for any $\eta>0$ there exists $C(n, \eta,\tau_1,\tau_2)>0$ such that
$$ \|g_1\odot_{CK} g_2({\bf x})\|_{A_{\rho, (1+\eta)(n+\delta+1)^\rho (\tau_1+\tau_2)}}\leq C(n, \eta,\tau_1,\tau_2) \|g_1\|_{\rho,\tau_1} \|g_2\|_{\rho,\tau_2}. $$
\end{lemma}
\begin{proof}
Choosing $s=(n+\delta)$ in Lemma \ref{l11}, we have
\[
\begin{split}
& \|g_1\odot_{CK} g_2({\bf x})\|_{\rho, (1+\eta)(n+\delta+1)^\rho (\tau_1+\tau_2)}  \\
&\leq 2^n \sup_{{\bf x}\in\mathbb R^{n+1}} \left( \sum_{|{\bf m}|=0}^\infty \sum_{|{\bf \ell}|=0}^{\infty} \|V_{{\bf m}+{\bf \ell}} ({\bf x})\| \|a_m\|\|b_\ell\| \exp(-(1+\eta)(n+\delta+1)^\rho(\tau_1+\tau_2) \|{\bf x}\|^{\rho({\bf x})})\right)\\
& \leq C'(n,\eta,\tau_1,\tau_2)\|g_1\|_{\rho,\tau_1}\|g_2\|_{\rho,\tau_2} \sum_{|{\bf m}|=0}^\infty \sum_{|\ell|=0}^{\infty} \|x\|^{|{\bf m}|+|\ell|} \frac{c(n,{\bf m})} {(n+\delta)^{|{\bf m}|}\|{\bf x}\|^{|{\bf m}|}} \frac{c(n,{\bf \ell})} { (n+\delta)^{|{\bf \ell}|} \|{\bf x}\|^{|\ell|}} \\
& \leq C(n,\eta,\tau_1,\tau_2)\|g_1\|_{\rho,\tau_1}\|g_2\|_{\rho,\tau_2},
\end{split}
\]
where in the last inequality we used the fact that
\[
\sum_{|{\bf m}|=0}^{\infty} \frac{c(n,{\bf m})}{(n+\delta)^{|{\bf m}|}}< \infty
\]
since
\[
\limsup_{r\to\infty}\left(\frac {1}{(n+\delta)^r}\sum_{|m|=r} c(n,{\bf m}) \right)^{\frac 1r}=\frac n{(n+\delta)}<1.
\]
\end{proof}
The results of this section will be used to characterize
 continuous homomorphisms in terms of differential operators in the sense that will be specified in the next session.

\section{Differential operators, representations of continuous homomorphisms} \label{SEC3}
Before studying continuous homomorphisms between $A_{\rho_i}$ $(i=1,\, 2)$ and those between $A_{\rho_i,+0}$ $(i=1,\, 2)$, we define a differential operator representation of homomorphisms from
$$\mathbb R_n [\vx]:=\left\{\sum_{|\vm|=0}^q V_{\vm}(\vx) a_{\vm}:\, a_{\vm}\in \mathbb R_n \quad\textrm{and}\quad q\in\mathbb N_0 \right\}$$
to
$$\mathbb R_n [[\vx]] :=\left\{\sum_{|\vm|=0}^\infty V_{\vm}(\vx) a_{\vm}:\, a_{\vm}\in \mathbb R_n\right\}. $$
We define the space of formal right linear differential operators of infinite order with coefficients in $\mathbb R_n [[\vx]]$ by
$$ \hat{D}:=\left\{ P=\sum_{\vm\in\mathbb N^n_0} u_{\vm}(\vx)\odot_L\partial^\vm_\vx:\, u_\vm(\vx)\in \mathbb R_n [[\vx]] \right\}.$$
Note that $\hat{D}$ is linearly isomorphic to $\prod_{\vm\in\mathbb N^n} \mathbb R_n [[\vx]]$, by the correspondence
$$ \sum_{\vm\in\mathbb N^n_0} u_{\vm}(\vx) \odot_L\partial^\vm_\vx\mapsto (u_{\vm}(\vx))_{\vm\in\mathbb N^n}.$$
\begin{proposition}\label{pnova1}
There are two left linear isomorphisms:
$$\hat{D}\leftrightarrow \operatorname{Hom}_{\mathbb R_n}(\mathbb R_n [\vx],\, \mathbb R_n [[\vx]])\leftrightarrow \prod_{\vm\in\mathbb N^n} \mathbb R_n [[\vx]],$$
where the first and second mappings are given by
\[
\begin{split}
& \sum_{\vm\in\mathbb N^n_0} u_{\vm}(\vx)\odot_L\partial^\vm_\vx\mapsto\left ( \mathbb R_n[\vx]\ni\sum_{\vm} V_{\vp}(\vx) f_{\vp}\mapsto \sum_{\vm\leq \vp}  \frac{\vp !}{(\vp-\vm)!} u_{\vm} (\vx)\odot_L V_{\vp-\vm} (\vx) f_{\vp} \right ),\\
& F\mapsto (F(V_{\vm} (\vx))/\vm!)_{\vm\in\mathbb N^n_0},
\end{split}
\]
respectively.
\begin{proof}
We can easily see that both mappings are injective left linear mappings between vector spaces and that their composition is given by
$$ \sum_{\vm} u_{\vm}(\vx) \odot_L \partial^\vm_\vx\mapsto \left( \sum_{\vm\leq \vp} u_{\vm}(\vx) \odot_L \frac{V_{\vp-\vm}(\vx)}{(\vp-\vm)!}\right)_{\vp\in\mathbb N^n_0}. $$
Therefore, it suffices to show that the relations
\begin{equation}\label{fnova1}
b_\vp(\vx)=\sum_{\vm\leq\vp} u_{\vm}(\vx) \odot_L \frac{V_{\vp-\vm}(\vx)}{(\vp-\vm)!},\quad \vp\in\mathbb N^n_0,
\end{equation}
induce a bijection
$$ \prod_{\vm\in\mathbb N^n_0} \mathbb R_n [[\vx]] \ni(u_{\vm}(\vx))_{\vm\in\mathbb N^n}\mapsto (b_\vp(\vx))_{\vp\in\mathbb N^n} \in \prod_{\vp\in\mathbb N^n_0} \mathbb R_n[[\vx]]. $$
This follows from the fact that the relation \eqref{fnova1} can be inverted as
 \begin{equation}\label{fnova2}
u_{\vm}(\vx)=\sum_{\vp\leq\vm} b_{\vp}(\vx) \odot_L \frac{V_{\vm-\vp}(-\vx)}{(\vm-\vp)!}.
\end{equation}
In fact, we can calculate the $\vp$-element of the image of $(u_{\vm}(\vx))_{\vm\in\mathbb N_0^n}$ by composition of \eqref{fnova2} and \eqref{fnova1} as
\[
\begin{split}
\sum_{\vs\leq \vp} \sum_{\vm\leq \vs}  \left( u_{\vm}(\vx)\odot_L  \frac{V_{\vs-\vm}(\vx)}{(\vs-\vm)!} \right) & \odot_L \frac{V_{\vp-\vs}(-\vx)}{(\vp-\vs)!}  =\sum_{\vs\leq \vp} \sum_{\vm\leq \vs}  u_{\vm}(\vx)\odot_L \left( \frac{V_{\vs-\vm}(\vx)}{(\vs-\vm)!} \odot_L \frac{V_{\vp-\vs}(-\vx)}{(\vp-\vs)!} \right)\\
&=\sum_{\vm\leq\vp } u_{\vm}(\vx) \odot_L \left( \sum_{\vm\leq\vs\leq\vp} \frac{V_{\vs-\vm}(\vx)\odot_L V_{\vp-\vs}(-\vx)}{(\vs-\vm)!(\vp-\vs)!} \right)\\
&=\sum_{\vm\leq\vp } u_{\vm}(\vx) \odot_L \left( \sum_{\vm\leq\vs\leq\vp} \frac{(\vp-\vm)! (-1)^{|\vp-\vs|} V_{\vp-\vm}(\vx) }{(\vs-\vm)!(\vp-\vs)!(\vp-\vm)!} \right) \\
&=\sum_{\vm\leq\vp } u_{\vm}(\vx) \odot_L \left( \frac{(1-1)^{|\vp-\vm|} V_{\vp-\vm}(\vx) }{(\vp-\vm)!} \right)=u_{\vp}(\vx), \\
\end{split}
\]
which implies that the composition is the identity. We can similarly show that the composition of \eqref{fnova1} and \eqref{fnova2} is the identity.
\end{proof}
\end{proposition}
Now we study continuous homomorphism from $A_{\rho_1}$ and $A_{\rho_2}$ and those from $A_{\rho_1,+0}$ and $A_{\rho_2,+0}$ where $\rho_i(r)$ ($i=1,\,2$) are two proximate orders for positive orders $\rho_i=\lim_{r\to\infty}\rho_i(r)>0$, satisfying
\begin{equation}\label{supernova}
r^{\rho_1(r)}=O(r^{\rho_2(r)}),\quad\textrm{as $r\to\infty$.}
\end{equation}
\begin{definition}\label{dnova1}
Let $\rho_i$ ($i=1,\, 2$) be two proximate orders for orders $\rho_i>0$ satisfying \eqref{supernova}. We take normalization $\hat{\rho_1}$ of $\rho_1$ as in Definition \ref{norm} and $G_{\hat{\rho}_1,q}$ by \eqref{c1}. We denote by $\mathbf{D}_{\rho_1\to\rho_2}$ and by $\mathbf{D}_{\rho_1\to\rho_2,0}$ the sets of all formal right linear differential operator $P$ of the form
$$ P=\sum_{\bfm\in\mathbb N^n}u_\bfm(\bfx) \odot_L \partial_{\bfx}^{\bfm} $$
where the multisequence $(u_\bfm(\bfx))_{\bfm\in\mathbb N^n_0}\subset A_{\rho_2}$ satisfies
\begin{equation}\label{enova1}
\forall\lambda>0,\, \exists \sigma>0,\, \exists C>0,\, \forall \bfm\in\mathbb N^n_0,\, \|u_\bfm\|_{\rho_2,\sigma}\leq C\frac {G_{\hat{\rho}_1,|\bfm|}}{\bfm!}\lambda^{|\bfm|}
\end{equation}
and
\begin{equation}\label{enova2}
\forall\sigma>0,\, \exists \lambda>0,\, \exists C>0,\, \forall \bfm\in\mathbb N^n_0,\, \|u_\bfm\|_{\rho_2,\sigma}\leq C\frac {G_{\hat{\rho}_1,|\bfm|}}{\bfm!}\lambda^{|\bfm|},
\end{equation}
respectively. Note that in the latter case, each $u_\bfm$ belongs to $A_{\rho_2,+0}$.
\end{definition}

For the following remark see also \cite[Remark $4.4$]{AIO20}.

\begin{remark}\label{rnova1}
By adding $\ln(c)/\ln(r)$ for a constant $c>0$ (with a suitable modification near $r=0$) to a proximate order $\rho_2(r)$ with order $\rho_2$, we get a new proximate order $\tilde{\rho}_2(r)$ for the same order $\rho_2$ satisfying
$$
\tilde{\rho}_2(r)=\rho_2(r)+\ln(c)/\ln(r)
$$
 for $r>1$, that is,
 $$
 r^{\tilde{\rho}_2(r)}=cr^{\rho_2(r)},
 $$
 eventually.
 Then $\|\cdot\|_{\tilde{\rho}_2,\sigma}$ and $\|\cdot\|_{\rho_2,c\sigma}$ become equivalent norms for $\sigma>0$, and the spaces $A_{\tilde{\rho}_2}$ and $A_{\tilde{\rho}_2,+0}$ are homeomorphic to $A_{\rho_2}$ and $A_{\rho_2,+0}$ respectively. By taking $c$ sufficiently large, we can take $\tilde{\rho}_2$ as
$$ \rho_1(r)\leq \tilde{\rho}_2(r), \quad\textrm{ for $r\geq r_0$} $$
for a suitable $r_0$, and we can choose normalizations $\hat{\rho}_1$ of $\rho_1$ and $\hat{\rho}_2$ of $\tilde{\rho}_2$ as
\newline
\newline
$(i)$ $\hat{\rho}_1(r)\leq \hat{\rho}_2(r)$ for $r\geq 0$.
\newline
\newline
Since a proximate order and its normalization define equivalent norms, we have
\newline
\newline
$(ii)$ $\| \cdot \|_{\hat{\rho}_2,\sigma}$ and $\| \cdot \|_{\rho_2,c\sigma}$ are equivalent for any $c>0$,
\newline
$(iii)$ $\mathbf{D}_{\rho_1\to\rho_2}=\mathbf{D}_{\hat{\rho}_1\to\hat{\rho}_2}$, $\mathbf{D}_{\rho_1\to\rho_2, 0}=\mathbf{D}_{\hat{\rho}_1\to\hat{\rho}_2,0}$.
\newline
\newline
Note further that Theorems \ref{main1} and \ref{main2} below are not affected by the replacement of $\rho_1$ and $\rho_2$ by $\hat{\rho}_1$ and $\hat{\rho}_2$.
\end{remark}
Now we will prove:
\begin{theorem}\label{main1}
Let $\rho_i(r)$ ($i=1,\, 2$) be two proximate orders for orders $\rho_i>0$ satisfying \eqref{supernova}.
\begin{enumerate}
\item Suppose that $P=\sum_{\bfm\in\mathbb N^n}u_\bfm(\bfx) \odot_L \partial_{\bfx}^{\bfm}\in\mathbf{D}_{\rho_1\to\rho_2}$. For a left monogenic entire function $f\in A_{\rho_1}$,
$$ Pf=\sum_{\bfm\in\mathbb N^n}u_\bfm(\bfx) \odot_L \partial_{\bfx}^{\bfm} f $$
converges and $Pf\in A_{\rho_2}$. Moreover, $f\mapsto Pf$ defines a continuous right linear homomorphism $P: A_{\rho_1}\to A_{\rho_2}$.
\item Let $F:A_{\rho_1}\to A_{\rho_2}$ be a continuous right linear homomorphism. Then there is a unique $P\in\mathbf{D}_{\rho_1\to\rho_2}$ such that $Ff=Pf$ holds for any $f\in A_{\rho_1}$.
\end{enumerate}
\end{theorem}
\begin{proof}
We can replace $\rho_i$ by $\hat{\rho}_i$ ($i=1,\, 2$) as in Remark \ref{rnova1}, and we may assume from the beginning that $\rho_i$ $(i=1,\, 2)$ are normalized proximate orders satisfying
$$ \rho_1(r)\leq \rho_2(r),\quad \textrm{for $r\geq 0$}, $$
which implies
 \begin{equation}\label{esupersupernova1}
 \|f\|_{\rho_2,\tau}\leq \|f\|_{\rho_1,\tau}
 \end{equation}
 for any $f$ and $\tau>0$.

$(1)$ We fix $\eta>0$ and $C_\eta>0$ in a such a way that
$$  (cr)^{\rho_2(cr)} \leq (1+\eta) c^{\rho_2} r^{\rho_2(r)}+C_\eta $$
see \cite[p. $16$, Proposition $1.20$]{LG86}. Using Lemma \ref{lnova1}, Lemma \ref{l12}, Remark \ref{rsupersupernovissima1} and the estimate \eqref{enova1}  for $(u_{\mathbf{m}}(\mathbf{x}))_{\mathbf{m}\in\mathbb N^n}$ in Definition \ref{dnova1}, we have a constant $k=k(\rho_1)$ depending only on $\rho_1$ such that for any $\epsilon>0$, $\tau>0$ there exists positive constants $\sigma(\epsilon)$, $C(n,\eta,\sigma(\epsilon), k\tau)$ and $C'(n,\eta,\sigma(\epsilon), k\tau)$ with the estimate
\begin{equation}\label{enovissima1}
\begin{split}
&\sum_{\mathbf{m}\in\mathbb N^n_0}\|u_{\mathbf{m}} \odot_L \partial^{\mathbf{m}}_{\mathbf{x}} f\|_{\rho_2, (1+\eta)(n+\delta+1)^{\rho_2}(\sigma(\epsilon)+k\tau)}\\
&\leq C(n,\eta,\sigma(\epsilon), k\tau) \sum_{\mathbf{m}\in\mathbb N^n_0} \|u_{\mathbf{m}}\|_{\rho_2,\sigma(\epsilon)} \|\partial^{\mathbf{m}}_{\mathbf{x}} f\|_{\rho_2,k\tau}\\
&\leq C(n,\eta,\sigma(\epsilon), k\tau) \sum_{\mathbf{m}\in\mathbb N^n_0} \|u_{\mathbf{m}}\|_{\rho_2,\sigma(\epsilon)} \|\partial^{\mathbf{m}}_{\mathbf{x}} f\|_{\rho_1,k\tau}\\
&\leq C'(n,\eta,\sigma(\epsilon), k\tau) \sum_{\mathbf{m}\in\mathbb N^n_0} \frac{G_{\rho_1,|\mathbf{m}|}}{\mathbf{m}!} \epsilon^{|\mathbf{m}|} \|f\|_{\rho_1,\tau} \frac{\mathbf{m}!}{G_{\rho_1,|\mathbf{m}|}}(2k(c(n,\vm))^{\rho_1/|\vm|} \tau)^{|\mathbf{m}|/\rho_1}\\
&= C'(n,\eta,\sigma(\epsilon), k\tau) \|f\|_{\rho,\tau} \sum_{q\in\mathbb N_0} \sum_{|\mathbf{m}|=q} \epsilon^q(2k(C(n))^{\rho_1/q} \tau)^{q/\rho_1}\\
&\leq C'(n,\eta,\sigma(\epsilon), k\tau) \|f\|_{\rho,\tau}  \sum_{q=0}^{+\infty} (q+1)^{n-1} \epsilon^q(2k(C(n))^{\rho_1/q} \tau)^{q/\rho_1}.
\end{split}
\end{equation}
In view of Remark \ref{rsupersupernovissima1}, the last sum converges if $\epsilon<(C(n)(2k\tau)^{1/\rho_1})^{-1}$. For such a choice of $\epsilon>0$ depending on $\rho_1$ and $\tau$, we set
$$
\tau'(\tau):= (1+\eta)(n+\delta+1)^{\rho_2}(\sigma(\epsilon)+k\tau)
$$
 and
$$ C''= C'(n,\eta,\sigma(\epsilon), k\tau) \sum_{q=0}^{+\infty} (q+1)^{n-1} \epsilon^q(2k(C(n))^{\rho_1} \tau)^{q/\rho_1}. $$
Then, $\sum_{|\mathbf{m}|} u_{\mathbf{m}}(\mathbf{x})\odot_L \partial^{\mathbf{m}}_{\mathbf{x}} f$ converges in $A_{\rho_2,\tau'(\tau)}$ and defines an element $Pf\in A_{\rho_2,\tau'(\tau)}$ satisfying
$$ \|Pf\|_{\rho_2,\tau'(\tau)}\leq C''\|f\|_{\rho_1,\tau}. $$
Since $f\in A_{\rho_1,\tau}$ was chosen arbitrarily, this implies the well-definedness and the continuity of $P:A_{\rho_1,\tau}\to A_{\rho_2,\tau'(\tau)}$. Also since $\tau>0$ was chosen arbitrarily, we get the well-definedness and the continuity of $P:A_{\rho_1}\to A_{\rho_2}$ by the definition of the inductive limit of locally convex spaces.

$(2)$ Let $F: A_{\rho_1}\to A_{\rho_2}$ be a continuous right linear homomorphism. Then, thanks to the theory of locally convex spaces, we can conclude, using Lemma \ref{lsupernova1}, that for any $\tau>0$, there exists $\tau'=\tau'(\tau)>0$ such that $F(A_{\rho_1,\tau})\subset A_{\rho_2,\tau'(\tau)}$ and that
$$ F: A_{\rho_1,\tau}\to A_{\rho_2,\tau'(\tau)} $$
is continuous. Therefore, we can in particular take $C(\tau)$ depending on $\tau>0$ for which
\begin{equation}\label{fsupernova1}
\|F f\|_{\rho_2,\tau'(\tau)}\leq C(\tau)\|f\|_{\rho_1,\tau},\quad\textrm{for any $f\in A_{\rho_1,\tau}$}.
\end{equation}
Let us define a multi-sequence of entire functions $(u_{\vm}(\vx))_{\vm\in\mathbb N^n_0}$ by
$$ u_{\vm}(\vx)=\sum_{\vs\leq \vm}\frac{F(V_{\vs}(\vx)) \odot_L V_{\vm-\vs}(-\vx)}{(\vm-\vs)! \vm!}, $$
whose convergence in $A_{\rho_2}$ will be proved together with their estimates. We define a formal differential operator $P\in \hat{D}$ of infinite order by
\begin{equation}\label{dnovissima2}
P=\sum_{\vm} u_{\vm}(\vx)\odot_L\partial^\vm_\vx .
\end{equation}
First we show that $P\in\mathbf{D}_{\rho_1\to\rho_2}$. For any fixed $\tau_0$ and $\tau_1$ with $0<\tau_0<\tau_1$, we have
\begin{equation}\label{enovissima2}
\begin{split}
\|u_{\vm} & \|_{\rho_2, (1+\eta)(n+\delta+1)^{\rho_1}(\tau_1+\tau'(\tau_1))}\\
& \leq C(\eta, n,\tau_1,\tau'(\tau_1)) \sum_{\vs\leq \vm}\frac{\| V_{\vm-\vs} (\vx)\|_{\rho_2,\tau_1}  \|F(V_{\vs}(\vx))\|_{\rho_2,\tau'(\tau_1)}}{(\vm-\vs)!\vs!}\\
&  \leq C(\eta, n,\tau_1,\tau'(\tau_1)) \sum_{\vs\leq \vm}\frac{\| V_{\vm-\vs} (\vx)\|_{\rho_1,\tau_1} C(\tau_1) \|V_{\vs}(\vx)\|_{\rho_1,\tau_1}}{(\vm-\vs)!\vs!}\\
& \leq C(\tau_1)C(\eta, n,\tau_1,\tau'(\tau_1)) \sum_{\vs\leq \vm} \frac{C(\tau_0,\tau_1)\tau_0^{-|\vm-\vs|/\rho_1} G_{\rho_1,|\vm-\vs|} C(\tau_0,\tau_1) \tau_0^{-|\vs|/\rho_1} G_{\rho_1,|\vs|}}{(\vm-\vs)! \vs!}\\
&\leq C(\tau_1) C(\eta, n,\tau_1,\tau'(\tau_1))  C(\tau_0,\tau_1)^2 \sum_{\vs\leq\vm} \binom{\vm}{\vs}  \frac{G_{\rho_1,|\vm|}}{\vm !}\tau_0^{-|\vm|/\rho_1}\\
&= C(\tau_1) C(\eta, n,\tau_1,\tau'(\tau_1)) C(\tau_0,\tau_1)^2 \frac{G_{\rho_1,|\vm|}}{\vm!} 2^{|\vm|} \tau_0^{-|\vm|/\rho_1}.
\end{split}
\end{equation}
Here we used Lemma \ref{l12} at the first inequality, \eqref{esupersupernova1} and \eqref{fsupernova1} at the second inequality, Lemma \ref{lsupersupernova1} with $0<\tau_0<\tau_1$ at the third inequality, and Lemma \ref{lsupersupernova2} at the fourth inequality. For a given $\epsilon>0$, we can take $\tau_0>0$ large enough such that
$$ 2\tau_0^{-1/\rho_1}<\epsilon. $$
Then, by choosing $\tau_1$ as $\tau_1>\tau_0$ and by putting
$$
\sigma:=(1+\eta)(n+\delta+1)^{\rho_1}(\tau_1+\tau'(\tau_1)),
$$
 we have
$$
\|u_\vm\|_{\rho_2,\sigma}\leq C'\frac{G_{\rho_1, |\vm|}}{\vm!}\epsilon^{|\vm|}
$$
for any $\vm$, which implies $P\in\mathbf{D}_{\rho_1\to\rho_2}$.

Now we show that $Pf=Ff$ for any $f\in A_{\rho_1}$. First we show that the equality holds for $f\in \mathbb R_n [\vx]$. This will imply the equality holds for any $f\in A_{\rho_1}$ since $\mathbb R_n[\vx]$ is dense in $A_{\rho_1}$ by Proposition \ref{psupersupernova1} and $P$ and $F$ are continuous. If $f\in \mathbb R_n [\vx]$ then there exists $m\in\mathbb N_0$ and $b_{\vs}\in \mathbb R_n$ for $|\vs|\leq m$ such that $f(\vx)=\sum_{|\vs|\leq m} V_{\vs}(\vx) b_\vs$. We have that
$$ F(f(\vx))=\sum_{|\vs|\leq m} F(V_{\vs} (\vx)) b_\vs $$
and
\begin{equation}\label{enovissima4}
\begin{split}
P(f(\vx))& =\sum_{|\vm|\leq m} u_\vm (\vx)\odot_L \partial_\vx^\vm f(\vx)\\
&=\sum_{|\vm|\leq m} u_\vm (\vx)\odot_L \partial_\vx^\vm \left( \sum_{|\vs|\leq m} V_{\vs}(\vx) b_\vs\right)\\
&=\sum_{|\vs|\leq m} \sum_{\vm\leq\vs} u_{\vm}(\vx)\odot_L \partial^\vm_\vx(V_\vs(\vx)) b_\vs\\
&=\sum_{|\vs|\leq m} \sum_{\vm\leq\vs} \sum_{\vp\leq \vm} \frac{\vs ! F(V_\vp(\vx))\odot_L V_{\vm-\vp}(-\vx)\odot_L V_{\vs-\vm}(\vx)}{(\vm-\vp)!(\vs-\vm)!\vp!}  b_\vs\\
&=\sum_{|\vs|\leq m} \sum_{\vp\leq \vs} F(V_{\vp}(\vx)) \odot_L\sum_{0\leq \vm-\vp\leq \vs-\vp} \frac{\vs! (\vs-\vp)! (-1)^{|\vm-\vp|} V_{\vs-\vp}(\vx)}{(\vs-\vp)!(\vm-\vp)! \vp!(\vs-\vm)!} b_\vs\\
&=\sum_{|\vs|\leq\vm} \sum_{\vp\leq \vs} F(V_\vp(\vx))\odot_L \frac{\vs! (1-1)^{|\vs-\vp|}}{(\vs-\vp)! \vp!} V_{\vs-\vp}(\vx) b_\vs\\
&=\sum_{|\vs|\leq m} F(V_\vs(\vx)) b_\vs
\end{split}
\end{equation}
thus $F(f(\bfx))=P(f(\bfx)).$
\end{proof}
\begin{theorem}\label{main2}
Let $\rho_i(r)$ ($i=1,\, 2$) be two proximate orders for orders $\rho_i>0$ satisfying \eqref{supernova}.
\begin{enumerate}
\item Suppose that $P=\sum_{\bfm\in\mathbb N^n_0}u_\bfm(\bfx) \odot_L \partial_{\bfx}^{\bfm}\in\mathbf{D}_{\rho_1\to\rho_2,0}$. For a left monogenic entire function $f\in A_{\rho_1,+0}$,
$$ Pf=\sum_{\bfm\in\mathbb N^n_0}u_\bfm(\bfx) \odot_L \partial_{\bfx}^{\bfm} f $$
converges and $Pf\in A_{\rho_2,+0}$. Moreover, $f\mapsto Pf$ defines a continuous homomorphism $P: A_{\rho_1,+0}\to A_{\rho_2,+0}$.
\item Let $F:A_{\rho_1,+0}\to A_{\rho_2,+0}$ be a continuous right linear homomorphism. Then there is a unique $P\in\mathbf{D}_{\rho_1\to\rho_2,0}$ such that $Ff=Pf$ holds for any $f\in A_{\rho_1,+0}$.
\end{enumerate}
\end{theorem}
\begin{proof}
Again we can make a substitution of proximate orders as in Remark \ref{rnova1}, and we may assume from the beginning that $\rho_i$ $(i=1,\, 2)$ are normalized proximate orders satisfying \eqref{esupersupernova1}.

$(1)$ We assume condition \eqref{enova2} for $(u_\vm(\vx))_{\vm\in\mathbb N^n_0}$ and denote by $\lambda_\sigma$ the constant $\lambda$ given there according to $\sigma$. Similar computations as in \eqref{enovissima1} yield
\begin{equation}\label{enovissima2}
\begin{split}
&\sum_{\mathbf{m}\in\mathbb N^n_0}\|u_{\mathbf{m}} \odot_L \partial^{\mathbf{m}}_{\mathbf{x}} f\|_{\rho_2, (1+\eta)(n+\delta+1)^{\rho_2}(\sigma+k\tau)}\\
&\leq C(n,\eta,\sigma, k\tau) \sum_{\mathbf{m}\in\mathbb N^n_0} \|u_{\mathbf{m}}\|_{\rho_2,\sigma} \|\partial^{\mathbf{m}}_{\mathbf{x}} f\|_{\rho_2,k\tau}\\
&\leq C(n,\eta,\sigma, k\tau) \sum_{\mathbf{m}\in\mathbb N^n_0} \|u_{\mathbf{m}}\|_{\rho_2,\sigma} \|\partial^{\mathbf{m}}_{\mathbf{x}} f\|_{\rho_1,k\tau}\\
&\leq C'(n,\eta,\sigma, k\tau) \sum_{\mathbf{m}\in\mathbb N^n_0} \frac{G_{\rho_1,|\mathbf{m}|}}{\mathbf{m}!} \lambda_\sigma^{|\mathbf{m}|} \|f\|_{\rho_1,\tau} \frac{\mathbf{m}!}{G_{\rho_1,|\mathbf{m}|}}(2k(c(n,\vm))^{\rho_1/|\vm|} \tau)^{|\mathbf{m}|/\rho_1}\\
&= C'(n,\eta,\sigma, k\tau) \|f\|_{\rho,\tau} \sum_{q\in\mathbb N_0} \sum_{|\mathbf{m}|=q} \lambda_\sigma^q(2k(c(n,\vm))^{\rho_1/q} \tau)^{q/\rho_1}\\
&\leq C'(n,\eta,\sigma, k\tau) \|f\|_{\rho,\tau}  \sum_{q=0}^{+\infty} (q+1)^{n-1} ( C(n) (2k\tau)^{1/\rho_1} \lambda_\sigma)^{q},
\end{split}
\end{equation}
for any $\sigma,\, \tau>0,$ $f\in A_{\rho_1,+0}$. Here $k=k(\rho_1)$ depends only on $\rho_1$, and $C'(n,\eta,\sigma, k\tau)$ is independent of $f$. Note that the last sum is finite if $C(n) (2k\tau)^{1/\rho_1} \lambda_\sigma<1$. In fact, we may first choose $\sigma$ as
$$0<\sigma<\frac \epsilon{2(1+\eta)(n+\delta+1)^{\rho_2}},$$
which determines $\lambda_\sigma$, and then secondly choose $\tau>0$ as
$$ k\tau<\frac{\epsilon}{2(1+\eta)(n+\delta+1)^{\rho_2}}, \quad C(n) (2k\tau)^{1/\rho_1} \lambda_\sigma <1.$$
Therefore, for any $\epsilon>0$, there exist $C''(\epsilon)>0$ and $\tau(\epsilon)>0$ such that
$$ \|P(f)\|_{\rho_2, \epsilon}\leq C''(\epsilon) \|f\|_{\rho_1, \tau(\epsilon)},\quad \textrm{for $f\in A_{\rho,+0}$}, $$
which implies that $P:A_{\rho_1,+0}\to A_{\rho_2,+0}$ is continuous.

$(2)$ For a given continuous homomorphism $F:A_{\rho_1,+0}\to A_{\rho_2,+0}$, we construct
$$P=\sum_{\vm} u_\vm(\vx)\odot_L\partial_\vx^\vm$$
via \eqref{dnovissima2} in the same way as in the proof of Theorem \ref{main1} $(2)$. Let us show the convergence of $P$ in $A_{\rho_2,+0}$ together with the estimates. For any $\sigma'>0$, by continuity of $F$ there exists a $\tau_1>0$ and a $C(\tau_1)>0$ such that
\begin{equation}\label{extranova1}
\|F(f)\|_{\rho_2,\sigma'/2} \leq C(\tau_1)\|f\|_{\rho_1,\tau_1},\quad \textrm{for $f\in A_{\rho_1,+0}$.}
\end{equation}
We fix $\sigma>0$. Take $\tau_0$ with $0<\tau_0<\min\{\tau_1,\sigma'/2\}$ where $\sigma'=\frac{\sigma}{(1+\eta)(n+\delta+1)^{\rho_1}}$. Similarly to the proof of \eqref{enovissima2}, we have
\begin{equation}\label{enovissima3}
\begin{split}
\|u_{\vm}\|_{\rho_2,\sigma} & \leq \|u_{\vm}  \|_{\rho_2, (1+\eta)(n+\delta+1)^{\rho_1} \sigma'}\\
& \leq C(\eta, n,\tau_1,\tau'(\tau_1)) \sum_{\vs\leq \vm}\frac{\| V_{\vm-\vs} (\vx)\|_{\rho_2,\sigma'/2}  \|F(V_{\vs}(\vx))\|_{\rho_2,\sigma'/2}}{(\vm-\vs)!\vs!}\\
&  \leq C(\eta, n,\tau_1,\tau'(\tau_1)) \sum_{\vs\leq \vm}\frac{\| V_{\vm-\vs} (\vx)\|_{\rho_1,\sigma'/2} C(\tau_1) \|V_{\vs}(\vx)\|_{\rho_1,\tau_1}}{(\vm-\vs)!\vs!}\\
& \leq C(\eta, n,\tau_1,\tau'(\tau_1)) \sum_{\vs\leq \vm} \frac{C(\tau_0,\sigma'/2)\tau_0^{-|\vm-\vs|/\rho_1} G_{\rho_1,|\vm-\vs|} C(\tau_1) C(\tau_0,\tau_1) \tau_0^{-|\vs|/\rho_1} G_{\rho_1,|\vs|}}{(\vm-\vs)! \vs!}\\
&\leq C(\eta, n,\tau_1,\tau'(\tau_1)) C(\tau_1) C(\tau_0,\sigma'/2) C(\tau_0,\tau_1) \sum_{\vs\leq\vm} \binom{\vm}{\vs}  \frac{G_{\rho_1,|\vm|}}{\vm !}\tau_0^{-|\vm|/\rho_1}\\
&= C(\eta, n,\tau_1,\tau'(\tau_1)) C(\tau_1) C(\tau_0,\sigma'/2) C(\tau_0,\tau_1) \frac{G_{\rho_1,|\vm|}}{\vm!} 2^{|\vm|} \tau_0^{-|\vm|/\rho_1}.
\end{split}
\end{equation}
Here we used Lemma \ref{l12} at the first inequality, \eqref{esupersupernova1} and \eqref{extranova1} at the second, Lemma \ref{lsupersupernova1} twice with $0<\tau_0<\sigma'/2$ and $0<\tau_0<\tau_1$ at the third, and Lemma \ref{lsupersupernova2} at the fourth.

Therefore, by defining
$$ \lambda:=2\tau_0^{-1/\rho_1}, $$
we have
$$ \|u_\vm(\vx)\|_{\rho_2,\sigma} \leq C'\frac{G_{\rho_1,|\vm|}}{\vm!} \lambda^{|\vm|}$$
for any $\vm$. Since $\sigma>0$ can be chosen arbitrarily, this implies $P\in\mathbf{D}_{\rho_1\to\rho_2,0}$.

The remaining thing is to show $F(f)=P(f)$ for $f\in A_{\rho_1,+0}$. This can be done completely in the same way as in the case of normal type, that is, the equality for polynomial $f$ due to \eqref{enovissima4} can be extended to $A_{\rho_1,+0}$ by continuity, since polynomials form a dense subset of $A_{\rho_1,+0}$, for which we again refer to Proposition \ref{psupersupernova1}.
\end{proof}

\section{Declarations} The authors declare that there are no competing interests and that there is no conflict of interests. No data sets were used and no fundings have been received.

\end{document}